\documentclass[12pt]{article}
\usepackage{amssymb,amsmath,psfrag}

\usepackage{amsmath}
\usepackage{amsfonts}
\usepackage{amsthm}
\usepackage{amssymb}
\usepackage{amscd}
\newtheorem{definition}{Definition}
\newtheorem{proposition}{Proposition}

\newtheorem{corollary}{Corollary}
\newtheorem{lemma}{Lemma}
\newtheorem{theorem}{Theorem}

\usepackage{graphicx}
\begin{document}

\title{Characterization of signed Gauss paragraphs}

\author{Jos\'{e} Gregorio Rodr\'{i}guez-Nieto
 \footnote{ Escuela de Matem\'aticas,
Universidad Nacional de Colombia,
Medell\'in-Colombia,
jgrodrig@unal.edu.co.}}

\maketitle

\begin{abstract}

In this paper we use theory of embedded graphs on oriented and compact $PL$-surfaces to construct minimal realizations of signed Gauss paragraphs. We prove that the genus of the ambient surface of these minimal realizations can be seen as a function of the maximum number of  Carter's circles. For the case of signed Gauss words, we use a generating set of $H_1(S_w,\mathbb{Z})$, given in \cite{CaEl}, and the intersection pairing of immersed $PL$-normal curves to present a short solution of the signed Gauss word problem. Moreover, we define the join operation on signed Gauss paragraphs to produce signed Gauss words such that both can be realized on the same minimal genus $PL$-surface.


\end{abstract}
\section{Introduction}
The Gauss word problem was proposed by K. F. Gauss in the beginning of the XIX century when he was studying certain class of words $w$ with the property that each letter of $w$ appeared exactly twice. He noted that if we label the crossings points of an oriented normal curve $\gamma$ on the plane, then the word formed with the letters that we met when following $\gamma$ has the same characteristics as Gauss words . In $1991$, J. S. Carter \cite{Ca} introduced the concept of \textit{signed Gauss paragraphs} to classify the stable geotopy class of immersed curves on orientable and compact surfaces. The construction of these signed paragraphs is similar to the one of Gauss words, but Carter considered, not only the number of components of the normal curve, but also the way in which the curve meets itself. 

The work of Carter left a topological solution of the planarity problem for signed Gauss paragraphs. He gave an algorithm to find, from a signed Gauss paragraph $w$, an oriented normal curve $\gamma_w$ embedded in an minimum genus orientable surface $M_w$, such that $w$ is the signed Gauss paragraph of $\gamma_w$. In this case we say that $(M_w,\gamma_w)$ is a minimal realization of $w$. In $1993$, G. Cairns and D. Elton, see \cite{CaEl}, continued the work of Carter and they presented a combinatorial solution of a particular case,  known as the \textit{signed Gauss word problem}. Here, signed Gauss words mean signed Gauss paragraphs with one component.

In order to solve the signed Gauss word problem, G. Cairns and D. Elton constructed a generating set $\{\Gamma, \gamma_1,...,\gamma_n\}$ for $H_1(M_{w},\mathbb{Z})$, and they proved that all these curves are null-homologue if and only if, for each of them, there exist a Alexander numbering. Their proof is very elegant, but it is so cumbersome. Here, we use intersection homology of curves on orientable surfaces, instead of Alexander numbering, to determine when these curves are null-homologue, because we think that this way is more direct. The intersection homology, $\left\langle \varphi,\lambda \right\rangle$, of two homology classes $\varphi$ and $\lambda$ on an orientable and compact surface $\Sigma$, is a very well known invariant of $H_1(\Sigma,\mathbb{Z})$. This number $\left\langle \varphi,\lambda \right\rangle$ has been defined as the number of intersections of any two representative transverse curves of $\varphi$ and $\lambda$, naturally, depending on the way in which one of them intersects the other one. One of the most important property of this number is: 

\begin{theorem}
\label{homolo} If $\varphi \in H_1(\Sigma,\mathbb{Z})$ and the homology intersection between $\varphi$ and any $\lambda$ in $H_1(\Sigma,\mathbb{Z})$ is equal to zero, then  $\varphi$ is null-homologue. Moreover, if moreover $\Sigma$ is connected and closed, and $\{\gamma_1,...,\gamma_m\}$ is a generating set of $H_1(\Sigma,\mathbb{Z})$, then $\Sigma=\mathbb{S}^{2}$ if and only if $\left\langle \gamma_i,\gamma_j \right\rangle=0$, for every $i,j$.
\end{theorem}

Motivated by the previous theorem we give a brief study of homology intersection of oriented $PL$-normal curves immersed in orientable and compact $PL$-surfaces. After that, for a signed Gauss paragraph $w$ we define a $4$-valency graph $\gamma_w$, with extra information on its vertices, and we  use theory of embedded graphs to get a realization $(\Sigma,\gamma_w)$ of $w$. 

Let $\Omega^{(k)}_w$ denote the union of triangles of the $k^{th}$-barycentric subdivision of $\Sigma$ that meet $\gamma_w$. Then $\{(\Omega^{(w)}_w,\gamma_w)\}^{\infty}_{k=2}$ is an infinite collection of realization of $w$, such that the number of components of the boundary of any of these $PL$-surfaces, $\Omega^{(k)}_{w}$, is equal to the number of Carter's circles of $w$ as defined in \cite{Ca}. Therefore, by gluing discs to the boundary of any $\Omega^{(k)}_{w}$ we get a minimal realization $(S_w,\gamma^{(k)}_{w})$ of $w$. 

If we have a generating set of $H_1(S_w,\mathbb{Z})$, from the Theorem \ref{homolo}, we would obtain a solution of the signed Gauss paragraph problem, but we only have a generating set of $H_1(S_w,\mathbb{Z})$ for the case when $w$ is a signed Gauss words, see \cite{CaEl}. From that, we define the \textit{join} operation for signed Gauss paragraphs $w$ to produce signed Gauss words $\widehat{w}$ with the property that $H_1(S_w,\mathbb{Z})$ is isomorphic to $H_1(S_{\widehat{w}},\mathbb{Z})$.

This paper is organized as follows. In Section $2$ we review the definitions of graph realizations and orientable and compact $PL$-surfaces. Section $3$ gives the definition of oriented $PL$-normal curves and signed Gauss paragraphs. We describe a method different from the Carter's to construct a minimal realization $(S_w,\gamma_w)$ for a given signed Gauss paragraph $w$ and we study the processes of G. Cairns and D. Elton \cite{CaEl} to find a particular generating set of $H_1(S_w,\gamma_w)$. We also define the join function on the set of signed Gauss paragraphs, and using this map we connect the solutions of the signed Gauss paragraphs and the signed Gauss words problems. In Section $4$ we study properties of the homology intersection of oriented $PL$-normal curves on $PL$-surfaces. Section $5$ solves the signed Gauss word problem by using homology intersection theory. Moreover, we show that the solution presented by G. Cairns and D. Elton can be approximated by using this method.

\section{Preliminary and notation} 

In this section we give a brief introduction of graph theory and a review of orientable $PL$- surfaces. For more details see \cite{Wh} and \cite{FoMa}.

\subsection{Graphs}

A \textit{graph} $G$ is an ordered pair $(V(G),E(G))$, where $V(G)$ is a finite non-empty set of vertices and $E(G)$ is a set of pairs $uv$ of distinct elements $u,v$ in $V(G)$, called edges. If $x=uv \in E(G)$ we say that $v$ and $v$ are \textit{adjacent} or \textit{neighbors} vertices, and that vertex $u$ or $v$ and the edge $x$ are \textit{incident} with each other. The \textit{valency}, $d(v,G)$, of a vertex $v$ is the number of edges of $G$ incident to $v$.  

Now, we will consider ``graphs'' with ``edges'' of the form $vv$, called \textit{loops} and  \textit{multiple edges}, that are edges that appear more than once in $E(G)$.

A graph $H$ is said to be a \textit{sub-graph} of a graph $G$ if $V(H)\subset V(G)$ and $E(H)\subset E(G)$. Let $S\subset V(G)$, the \textit{induced sub-graph} $\left\langle S\right\rangle$ is the maximal sub-graph $H$ of $G$ with $V(H)=S$.     

\begin{definition}
A \textit{walk}, $P(u,w)$, in a graph $G$ is an ordered sequence of edges written as the linear combination $x_1+x_2+\cdots +x_p$, where $x_i=v_{i-1}v_i$, $i=1,2,...,n$, $u=v_1$ and $w=v_p$. The number $p$ is called the \textit{length} of the walk. A walk is said to be \textit{closed} if $v_n = v_0$, otherwise it is called \textit{open}. The walk is called a \textit{trail} if all its edges are distinct and a \textit{path} if all the vertices are distinct. 
\end{definition}

A graph $G$ is said to be \textit{connected} if for every couple of vertices $v$ and $w$, there exist a path $p(v,w)$. The minimal length of all paths, $p(v,w)$, is called the \textit{distance} between $v$ and $w$, denoted by $\rho_{G}(v,w)$.  

The \textit{first barycentric subdivision of a graph $G$} is obtained by adding a new vertex in each edge of $G$. We define the \textit{$k^{th}$-barycentric subdivision of $G$}, denoted $G^{(k)}$, as the first barycentric subdivision of $G^{(k-1)}$.

A \textit{digraph} is a graph with oriented edges. In this case we use the notation $[u,w]$ to denote the edge $uw$ oriented  from $u$ to $w$. Since, all the definitions and results given in this paper are true for both, graphs and digraphs, we will use the word digraph to refer, undistinguished, graphs and digraphs. 

\begin{definition}
Two digraphs $G=(V_G,E_G)$ and $H=(V_H,E_H)$ are said to be \textit{isomorphic} if there exists an one-to-one and onto function $\nu : V_G \rightarrow V_H$ with $\left(\left[u,w\right] \in E_G \Leftrightarrow \left[\nu(u),\nu(w)\right] \in E_H \right)$ and $\left(\left[u,w\right] \in E_H \Leftrightarrow \left[\nu^{-1}(u),\nu^{-1}(w)\right] \in E_G \right)$. 

They are said to be \textit{homeomorphic} if they have isomorphic barycentric subdivisions.
\end{definition}    

Let $G$ be a digraph and let $v$ a vertex of $G$. The \textit{star} of $G$ at $v$, denoted $ost(v,G)$, is the sub-digraph of $G$ whose vertices are $v$ and all its neighbors and its edges are those incident to $v$.  The \textit{closed star} of $G$ at $v$, denoted by $st(v,G)$, is the sub-digraph of $G$ induced by $v$ and all its neighbors. A \textit{star-digraph} is a digraph that is isomorphic to $ost(v,G)$, where $G$ is a digraph and $v$ is a vertex of $G$.  

\begin{definition}
A \textit{realization of a digraph $G$} on the sphere $\mathbb{S}^{2}$ is a pair $R(G)= (\stackrel{\rightarrow}{V}(G),\stackrel{\rightarrow}{E}(G))$, where $\stackrel{\rightarrow}{V}(G)$ is a set of points in $\mathbb{S}^{2}$ and $\stackrel{\rightarrow}{E}(G)$ is a set of oriented simple curves on $\mathbb{S}^{2}$, endowed with two one-to-one and onto functions $f:V(G) \rightarrow \stackrel{\rightarrow}{V}(G)$, with $f(v)=\stackrel{\rightarrow}{v}$ and $g:E(G) \rightarrow \stackrel{\rightarrow}{E}(G)$, with $g(e)=\stackrel{\rightarrow}{e}$, such that if $e=[v_i,v_j] \in E(G)$, then $\stackrel{\rightarrow}{e}$ is oriented from $\stackrel{\rightarrow}{v}_i$ to $\stackrel{\rightarrow}{v}_j$.

If every $\stackrel{\rightarrow}{e}_i$ and $\stackrel{\rightarrow}{e}_j$ are disjoint or only intersect at their ends points, we say that $R(G)$ is a \textit{planar realization} of $G$.
\end{definition}

If a digraph $G$ has a non planar realization $R(G)$, we assume, without lost of generality, that for every $\stackrel{\rightarrow}{e}_i, \stackrel{\rightarrow}{e}_j \in \stackrel{\rightarrow}{E}(G)$, the number of elements of $\stackrel{\rightarrow}{e}_i \cap \stackrel{\rightarrow}{e}_j$ is finite and every $c \in \stackrel{\rightarrow}{e}_i \cap \stackrel{\rightarrow}{e}_j$, not end points, is of the form $\begin{array}{c} \includegraphics[scale=0.15]{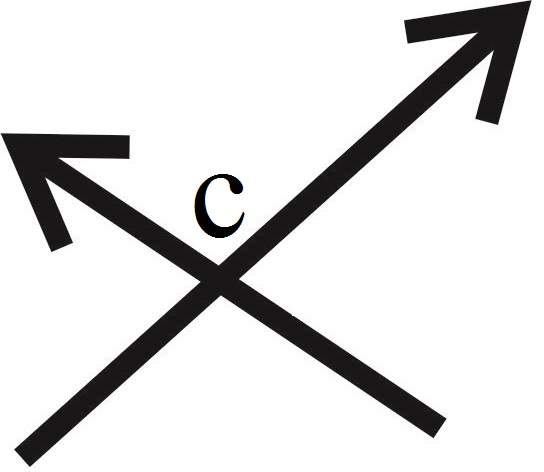} \end{array}$. These type of intersection points are called \textit{virtual crossings} and $R(G)$ is called \textit{virtual realization} of $G$. 

The virtual realization of a star-digraph is called \textit{spider-digraph}.

\begin{definition}
Let $\mathcal{S}$ be the set of all spider-digraphs, $G$ a digraph and $R(G)$ a virtual realization of $G$. A \textit{map form} on $R(G)$ is a function $f : \stackrel{\rightarrow}{V}(G) \rightarrow \mathcal{S}$ such that $f(\stackrel{\rightarrow}{v})\ddot{=}ost(\stackrel{\rightarrow}{v},R(G))$, where $\ddot{=}$ means equal up to isotopic transformations of $\mathbb{S}^{2}$ and $f$ induces an isomorphism between $ost(\stackrel{\rightarrow}{v},R(G))$ and $f(\stackrel{\rightarrow}{v})$.   

\end{definition}

Now, we will consider the following definition.

\begin{definition}
 For a digraph $G$, a \textit{virtual diagram} of $G$, denoted $D(G)$, is a pair $(R(G), f : \stackrel{\rightarrow}{V}(G) \rightarrow \mathcal{S} )$, where $R(G)$ is a virtual realization of $G$ and $f$ is a map form on $R(G)$.
\end{definition}

Let $D(G)$ be a virtual diagram $(R(G), f : \stackrel{\rightarrow}{V}(G) \rightarrow \mathcal{S} )$ of a graph $G$, and let $a$ and $b$ to be two non-vertices points on $R(G)$ living in the same edge $e(a,b)$ of $R(G)$. We take a sub-segment $A_{(a,b)}$ of $e(a,b)$ that joins the points $a$ and $b$. If $A_{(a,b)}$ only has virtual crossings or it does not have any crossing, then we remove $A_{(a,b)}$ from $e(a,b)$ and replace it by another segment $\widetilde{A}_{(a,b)}$ in $R(G)-A_{(a,b)}$ with the same end points such that if there would be intersecting points between $\widetilde{A}_{(a,b)}$ and $R(G)-A_{(a,b)}$, then they have to be virtual crossings. This process was called by L. Kauffman \cite{Ka} as \textit{detour moves}.

Note that detour moves neither change the set of vertices nor the map form of any virtual realization $R(G)$ of a digraph $G$.

\begin{definition}
Two virtual diagrams are isomorphic if and only if they can be changed into the other by using a finite sequence of detour moves.
\end{definition}

A \textit{flat virtual knot diagram} is a virtual diagram $(R(G),f_{i} : \stackrel{\rightarrow}{V}(G) \rightarrow \mathcal{S})$ such that for every $\stackrel{\rightarrow}{v}\in \stackrel{\rightarrow}{V}(G) $, $f(\stackrel{\rightarrow}{v})\ddot{=}$ 
$\begin{array}{c} \includegraphics[scale=0.7]{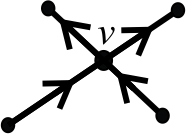} \end{array}$.

\subsection{$PL$-surfaces}

A \textit{$2$-polyhedron} is a compact subspace $\Sigma$ of $\mathbb{R}^{3}$ obtained by gluing finite triangles $\{\delta_1,...,\delta_k\}$ such that for every $i\neq j$, $\delta_i \cap \delta_j$ is either empty, or a common side, or a common vertex of $\delta_i$ and $\delta_j$.  The set of vertices of $\Sigma$, denoted by $V_{\Sigma}$, is called the $0$-skeleton of $\Sigma$, the set of edges of $\Sigma$, denoted by $\Gamma_{\Sigma}$, is called the $1$-skeleton of $\Sigma$ and the set of triangles of $\Sigma$, denoted by $F_{\Sigma}$, is called the $2$-skeleton of $\Sigma$.  

The \textit{closed star} of $\Sigma$ at $v$, denoted by $st(v,\Sigma)$, is the union of the triangles $\delta$ of $\Sigma$ having $v$ as a vertex. 

\begin{definition}
A \textit{$PL$-surface with holes} is a $2$-polyhedron such that the closed star of every vertex $v$ of $\Sigma$ is homeomorphic either to the closed $2$-disk with $v$ at the center or to the closed $2$-disk with $v$ on the boundary. These last type of vertices are called \textit{boundary vertices}.

If $\Sigma$ does not have boundary vertices, then $\Sigma$ is called \textit{closed $PL$-surface}.
\end{definition}
 
\begin{definition}
We say that a surface is orientable if we can choose an orientation of its edges and faces such that if an edge belongs to two adjacent faces then these faces travel the edge in opposite ways. 
\end{definition}	

In order to simplify the notation we will use the word \textit{surface} to refer both orientable $PL$-surface with holes or closed. 

An \textit{oriented path} on a surface $\Sigma$ from the vertex $v$ to the vertex $w$ is a path $p(v,w)$ in $\Gamma_{\Sigma}$ with edges coherently oriented. A \textit{$PL$-oriented circle} is an oriented closed path.

The \textit{link} of a vertex $v$ of $\Sigma$, denoted by $lk(v,\Sigma)$, is the path conformed by the edges opposite to $v$ in the triangles containing $v$. With this notation, $v$ is in the boundary of $\Sigma$ if and only if $lk(v,\Sigma)$ is a open path.  

The proof of the following theorem is well know, so we will omit it. 

\begin{theorem}
Let $\Sigma$ be a surface and $L_{\Sigma}$ be the set of boundary vertices of $\Sigma$. If $L_{\Sigma} \neq \emptyset$, then the boundary, $\partial_{\Sigma}$, of $\Sigma$ is conformed by a finite collection of $PL$-circles.
\end{theorem}  

For an orientable surface $\Sigma$ we define the interior of $\Sigma$, denoted $int(\Sigma)$, as $int(\Sigma)=\Sigma - \partial_{\Sigma}$.

\begin{definition}\cite{Wh}
\label{defwh} Let $\Gamma$ be a finite digraph, with $V(\Gamma)=\{v_1,...,v_p\}$ and $E(\Gamma)=\{x_1,...,x_q\}$. Let $\Sigma$ be a connected surface. An \textit{embedding} of $\Gamma$ in $\Sigma$ is a subspace, $\Gamma(\Sigma)$, of $\Sigma$,
\begin{center}
$\Gamma(\Sigma)=\bigcup^{p}_{i=1} v_i(\Sigma) \cup \bigcup^{q}_{j=1} x_j(\Sigma)$,
\end{center}
for which there exist injective and onto functions $f: V(\Gamma) \rightarrow \{v_1(\Sigma),...,v_p(\Sigma)\}$ and $g: E(\Gamma) \rightarrow \{x_1(\Sigma),...,x_p(\Sigma)\}$ such that:  

(1) $v_1(\Sigma)$,...,$v_p(\Sigma)$ are distinct points of $int(\Sigma)$.

(2) $x_1(\Sigma)$,...,$x_q(\Sigma)$ are mutually disjoint open paths in $int(\Sigma)$,

(3) $x_j(\Sigma) \cap v_i(\Sigma)=\emptyset $, $i=1,2,...,p$; $j=1,2,...,q$ and

(4) if $v_{i_1}$ and $v_{i_2}$ are the end points of $x_i$, then $v_{i_1}(\Sigma)$ and $v_{i_2}(\Sigma)$ are the end points of $x_i(\Sigma)$.  

(5) if $x_i$ is oriented from $v_{i_1}$ to $v_{i_2}$, then $x_i(\Sigma)$ is oriented from $v_{i_1}(\Sigma)$ to $v_{i_2}(\Sigma)$. 

The subspace $\Gamma(\Sigma)$ is called a \textit{realization of $\Gamma$ on $\Sigma$}.
\end{definition}

For a virtual diagram $(R(G),f: \stackrel{\rightarrow}{V}(G) \rightarrow \mathcal{S})$ and a surface $\Sigma$, we define a \textit{surface realization of $R(G)$ on $\Sigma$} as a subspace 
\begin{center}
$R(G)(\Sigma)=\bigcup^{p}_{i=1} \stackrel{\rightarrow}{v}_i(\Sigma) \cup \bigcup^{q}_{j=1} \stackrel{\rightarrow}{x}_j(\Sigma)$ 
\end{center}
of $\Sigma$ that satisfy the conditions of the Definition \ref{defwh}, endowed with a map form 
\begin{center}
$\widetilde{f} : \{\stackrel{\rightarrow}{v}_1(\Sigma),...,\stackrel{\rightarrow}{v}_p(\Sigma)\} \rightarrow \mathcal{S}$,
\end{center}
such that $\widetilde{f}(\stackrel{\rightarrow}{v}_i(\Sigma))=f(\stackrel{\rightarrow}{v}_i)$, for every $i=1,2,...,p$. 

The proof of the following theorem can be found in \cite{Wh} and \cite{Ka}. 

\begin{theorem}
\label{emb} Every virtual diagram has a surface realization. 
\end{theorem}

The genus, $g(\Gamma)$, of a digraph is the minimum genus among all the closed and connected surfaces on which $\Gamma$ can be embedded. The virtual genus, $g_v(R(G))$, of a virtual diagram $(R(G),f: \stackrel{\rightarrow}{V}(G) \rightarrow \mathcal{S})$ is the minimum genus among all the surfaces on with it can be realizable.  

The first barycentric subdivision of triangle is the change of this by six new triangles obtained by the construction of the three medians of $\delta$. The \textit{first barycentric subdivision of a surface $\Sigma$} is the result of the first barycentric subdivision of all the faces of $\Sigma$, it is denoted by $\Sigma^{(1)}$. By a recurrence process we define the \textit{$k^{th}$-barycentric subdivision of $\Sigma$}, denoted $\Sigma^{(k)}$, as the first barycentric subdivision of $\Sigma^{(k-1)}$.  

\begin{definition}
Let $\Sigma$ and $\widetilde{\Sigma}$ be two surfaces. A continuous function $f: \Sigma \rightarrow \widetilde{\Sigma}$ is called a \textit{simplicity function} if $f$ sends each face of $\Sigma$ onto a face of $\widetilde{\Sigma}$. We say that $\Sigma$ is \textit{embedded} in $\widetilde{\Sigma}$, if there exist barycentric subdivisions, $\Sigma^{(k)}$ and $\widetilde{\Sigma}^{(n)}$, of $\Sigma$ and $\widetilde{\Sigma}$, respectively, and a simplicity and injective function $g: \Sigma^{(k)} \rightarrow \widetilde{\Sigma}^{(n)}$. In the case that $g(\Sigma^{(k)})=\widetilde{\Sigma}^{(n)}$ and $g^{-1}$ is also simplicity, we say that $\Sigma$ and $\widetilde{\Sigma}$ are \textit{homeomorphic}. 
\end{definition}

Let $\Sigma$ and $\widetilde{\Sigma}$ be two surfaces. We said that an embedding $f: \Sigma^{(k)} \rightarrow \widetilde{\Sigma}^{(n)}$ preserves the orientation if $f$ preserves the orientation of edges and triangles.

\section{Signed Gauss paragraphs}
In this section we give a short introduction to the study of oriented normal curves on surfaces and signed Gauss paragraphs. We present an algorithm to construct an infinite collection of realizations for a given signed Gauss paragraph.

\subsection{$PL$-normal curves}

To start this section we give the following definition. 

\begin{definition}

\label{nor} Let $\Sigma$ be a surface. An oriented closed trail $\gamma$ in $\Gamma_{\Sigma}$ is called an \textit{oriented $PL$-normal curve} if the set of vertices $V_{\gamma}=\{v_{j_{i}} \mid i=1,2,...,p\}$ of $\gamma$ can be split into two disjoint subsets $C_{\gamma}=\{a_1,...,a_n\}$ and $O_{\gamma}=\{v\in V_{\gamma} \mid v \notin C_{\gamma} \} $ such that $d(a_i,\gamma)=4$, $i=1,2,...,n$ and, for every $v\in O_{\gamma}$, $d(v,\gamma)=2$. Moreover $\gamma \cap st(a_i,\Sigma)$ looks like the Figure \ref{cross}-$a$ shows.

\begin{figure}[ht]
\begin{center}
\includegraphics[scale=0.5]{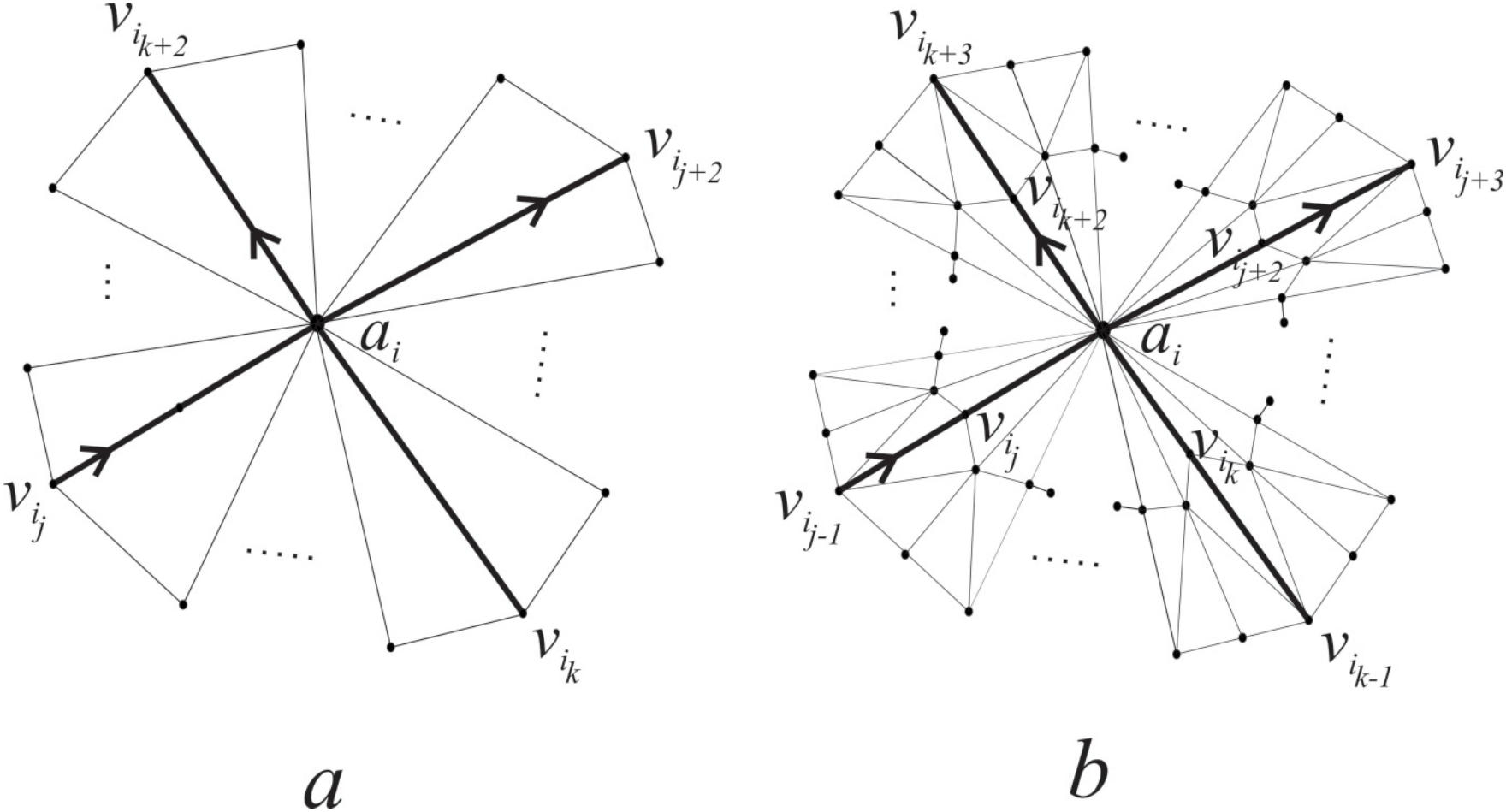}
\caption{Crossings points}
\label{cross}
\end{center}
\end{figure}

\end{definition}

We classify the elements of $V_{\gamma}$ as \textit{crossings points} if they belong to $C_{\gamma}$ or \textit{ordinary points} if they belong to $O_{\gamma}$ and we say that $\gamma$ travels the crossing point $a_i$ positively (negatively)  if this is traveled from the vertex $v_{i_{k}}$ (from $v_{i_{j}}$).

\begin{definition}
An \textit{oriented $PL$-normal curve of $k$ components} on a surface $\Sigma$ is a collection of $k$ oriented $PL$-normal curves $\beta=\gamma_1 \cup...\cup\gamma_k$ such that if $\gamma_i \cap \gamma_j \neq \emptyset$, then $\left|\gamma_i \cap \gamma_j\right|$ is finite and for every $c \in \gamma_i \cap \gamma_j$, $c$ is of the form described in Figure \ref{cross}.

We denote $C^{(\beta)}=\bigcup_{i\neq j} \gamma_i \cap \gamma_j$ and $C_{\beta}=C_{\gamma_1}\cup \cdots \cup C_{\gamma_k}\cup C^{(\beta)}$.
\end{definition}

If there are crossings points $a_i$ and $a_j$ in $C_{\beta}$ such that $\rho_{\beta}(a_i,a_j)=2$, then $\rho_{\beta^{(2)}}(a_i,a_j)=4$, see Figure \ref{cross}. Therefore we may assume without loss of generality that $\rho_{\beta}(a_i,a_j)>2$, for every $a_i$, $a_j$ in $C^{(\beta)}$.

We will denote the set of all $(\Sigma,\gamma)$, where $\gamma$ is an oriented $PL$-normal curve on a connected surface $\Sigma$, by $\mathcal{D}$. 

Let $(\Sigma,\gamma) \in \mathcal{D}$. Then, for every $k \in \mathbb{Z}$, the $k^{th}$ barycentric subdivision of $\Sigma$ induces a subdivision of $\gamma$, denoted by $\gamma^{(k)}$, and so $(\Sigma^{(k)},\gamma^{(k)})\in \mathcal{D}$. 

\begin{definition}
Let $(\Sigma,\gamma)$ and $(\widetilde{\Sigma},\lambda)$ in $\mathcal{D}$, 

(1) we say that $(\Sigma,\gamma)$ is \textit{geotopic} to $(\widetilde{\Sigma},\lambda)$, denoted $(\Sigma,\gamma) \approx (\widetilde{\Sigma},\lambda)$, if there exists a surface $Z$ and orientation preserving injective simplicity functions $f : \Sigma^{(k)} \rightarrow Z^{(m)}$ and $g:\widetilde{\Sigma}^{(n)} \rightarrow Z^{(m)}$ such that $f(\gamma^{(k)})=g(\lambda^{(n)})$.   

(2) We say that $(\Sigma,\gamma)$ is \textit{stably grotopic} to $(\widetilde{\Sigma},\lambda)$ if there exist a finite collection $(\Sigma_i,\gamma_i)\in \mathcal{D}$, $i=1,2,...,t$, such that 
\begin{center}
$(\Sigma,\gamma) \approx (\Sigma_{1},\gamma_{1}) \approx \cdots \approx (\Sigma_{t},\gamma_{t})\approx (\widetilde{\Sigma},\lambda)$.
\end{center}

\end{definition}

The geotopic relation is not a equivalence relation whereas stably geotopic is, see \cite{CaEl}.

\subsection{Signed Gauss paragraphs}

A \textit{signed paragraph} is a finite collection of words $\{w_1,...,w_n\}$ in some alphabet of the form $S=\{a_1,...,a_n,a^{-1}_1,...,a^{-1}_n\}$. A \textit{signed Gauss paragraph} is a signed paragraph such that every letter of $S$ occurs just once in a single word component $w_{j}$ of $w$.
  
\begin{definition}
Two signed Gauss paragraphs are said to be \textit{isomorphic} if one of them can be changed into the other by a finite sequence of the following transformations: (1) cyclic permutations of the letters of any word component of the paragraph, (2) change of alphabets and (3) reorganizing the components of paragraph.
\end{definition}

\underline{Constructing signed Gauss paragraphs:}
Let $(\Sigma,\gamma)\in \mathcal{D}$ be an oriented $PL$-normal curve of $n$ components, $\gamma_1,...,\gamma_n$, and $k$ crossings labeled with the letters $a_{1}$,...,$a_{k}$. We choose a component, $\gamma_{j}$, of $\gamma$ and an ordinary point $x_{j}$ on $\gamma_j$. Now, we follow the curve $\gamma_{j}$ writing down the list of crossing labels, denoted by $w_{\gamma_{j}}$, with the convention that we add $a_{i}$ (or $a^{-1}_{i}$) if the crossing is traveled positively (or negatively). 

If $n\geq 2$, the collection $w_{\gamma}=\{w_{\gamma_{1}},..,w_{\gamma_{n}}\}$ is called a \textit{signed Gauss paragraph of $\gamma$}, but if $n=1$, then $w_{\gamma}$ is called a \textit{signed Gauss word}.

With the notation above $w_{\gamma}=w_{\gamma^{(k)}}$ for every barycentric subdivision, $(\Sigma^{(k)},\gamma^{(k)})$, of $(\Sigma,\gamma)$. The proof of the following lemma can be found in \cite{Ca} and \cite{CaEl}.

\begin{lemma}
If $(\Sigma,\gamma)$ is stably geotopic to $(\widetilde{\Sigma},\lambda)$ if and only if $w_{\gamma}$ is isomorphic to $w_{\lambda}$.
\end{lemma}

\begin{definition}
A signed Gauss paragraph $w$ is said to be \textit{realizable} if there exists $(\Sigma,\gamma)\in \mathcal{D}$ such that $w_{\gamma}=w$. If $\Sigma=S^{2}$, then $w$ is called \textit{geometric signed Gauss paragraph}. 
\end{definition}

We assume without loss of generality that signed Gauss paragraphs are of the from $w=\{w_1,...,w_n\}$, with $w_i=a^{e_{i_{1}}}_{i_1}...a^{e_{i_{k(i)}}}_{i_{k(i)}}$, and every $w_i$ and $w_j$ are not disjoints. 

The following definition was taken from \cite{Ca}.

\begin{definition}[Carter circles]
\label{car} A \textit{Carter circle}, $c(w)$, of $w$ is a subset of 
\begin{center}
$E(w)=\bigcup^{n}_{i=1}\{\pm [a^{e_{i_{1}}}_{i_{1}},a^{e_{i_{2}}}_{i_{2}}],...,\pm [a^{e_{i_{k(i)}}}_{i_{k(i)}},a^{e_{i_{1}}}_{i_{1}}]\}$, 
\end{center}
such that
\begin{enumerate}
\item If $[a^{e_{t}}_{t},a^{e}] \in c(w)$, then $-sign(e)[a^{e_{g}}_{g},a^{e_{h}}_{h}] \in c(w)$, where $a^{e_{h}}_{h}=a^{-1}$ if $e=1$ and $a^{e_{g}}_{g}=a$ if $e=-1$,
\item If $-[a^{e},a^{e_{f}}_{f}] \in c(w)$, then $sign(e)[a^{e_{g}}_{g},a^{e_{h}}_{h}] \in c(w)$, where $a^{e_{h}}_{h}=a$ if $e=-1$ and $a^{e_{g}}_{g}=a^{-1}$ if $e=1$.
\end{enumerate}
\end{definition}

The next proposition resumes some properties of the Carter's circles.

\begin{proposition}
Let $B(w)$ be a set of  Carter's circles. Then,

(1) for every $c_1(w)$ and $c_2(w)$ in $B(w)$, $c_1(w) \cap c_2(w) =\emptyset$,

(2) $B$ is maximal if and only if $E(w)=\bigcup_{c(w)\in B(w)} c(w)$.

(2) Let $C(w)$ be the maximal set of Carter's circles, then the cardinal, $b(w)$, of $C(w)$ is invariant under isomorphisms of $w$.
\end{proposition}

Let $w$ be a signed Gauss paragraph and $\Gamma_{w}$ be the digraph such that $V(\Gamma_w)=\{a_{1},...,a_{n}\}$ and $E(\Gamma_w)=\bigcup^{n}_{i=1}\{[a^{e_{i_{1}}}_{i_{1}},a^{e_{i_{2}}}_{i_{2}}],...,[a^{e_{i_{k(i)}}}_{i_{k(i)}},a^{e_{i_{1}}}_{i_{1}}]\}$,  where $[a^{e_{i_{t}}}_{i_{t}},a^{e_{i_{t+1}}}_{i_{t+1}}]$ denotes the edge $a_s a_l$ oriented from $a_{s}$ to $a_{l}$ if and only if $a_{i_{t}}=a_{s}$ and $a_{i_{t+1}}=a_{l}$. 

Since, for every $a_i \in V(\Gamma_w)$ there are edges of the form $[c,a^{-1}_i]$, $[a^{-1}_i,b]$, $[a_i,d]$ and $[c,a_i]$, then we can construct a flat virtual knot diagram $(K_w=R(\Gamma_w),f: \stackrel{\rightarrow}{V}(\Gamma_w) \rightarrow \mathcal{S}$ of $\Gamma_w$ such that for every $\stackrel{\rightarrow}{a}_i \in \stackrel{\rightarrow}{V}(\Gamma_w)$. 

\begin{figure}[ht]
\begin{center}
$f(\stackrel{\rightarrow}{a}_i)\ddot{=}$ $\begin{array}{c}
 \includegraphics[scale=1.1]{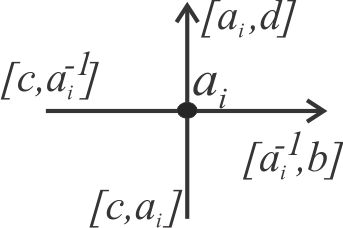}
\end{array}$
\caption{Vertex organizing}
\label{reali}
\end{center}
\end{figure} 

If $w=\{w_1,...,w_m\}$, with $w_i=a^{e_{i_{1}}}_{i_1}...a^{e_{i_{k(i)}}}_{i_{k(i)}}$, then $K_w$ has $m$ components $K_{w_{i}}$, $i=1,2,...,m$ and we write

\begin{center}
$K_{w_i}=[a^{e_{i_{1}}}_{i_1},a^{e_{i_{2}}}_{i_2}]+...+[a^{e_{i_{k(i)}}}_{i_{k(i)}},a^{e_{i_{1}}}_{i_1}]$.
\end{center} 

From Theorem \ref{emb}, there exist  a surface $\Sigma$ and a realization $D_w$ of $K_w$ on $\Sigma$, such that 
\begin{center}
$D_w= \bigcup^{m}_{i=1} a_i(\Sigma) \cup \bigcup^{m}_{i=1} \bigcup^{k(i)}_{j=1} [a^{e_{i_{j}}}_{i_j},a^{e_{i_{j+1}}}_{i_{j+1}}](\Sigma)$,
\end{center}  
where $a_{i}(\Sigma)$, $i=1,2,...,n$, are points of $\Sigma$ and $[a^{e_{i_{j}}}_{i_j},a^{e_{i_{j+1}}}_{i_{j+1}}](\Sigma)$  is an oriented simple arc on $\Sigma$, for $i=1,2,...,m$, $j=1,2,...,k(i)$, that satisfy the conditions given in the Definition \ref{defwh}. So, $D_w=D_{w_1}\cup \cdots \cup D_{w_m} $, where 
\begin{center}
$D_{w_i}=[a^{e_{i_{1}}}_{i_1},a^{e_{i_{2}}}_{i_2}](\Sigma)+...+[a^{e_{i_{k(i)}}}_{i_{k(i)}},a^{e_{i_{1}}}_{i_1}](\Sigma)$.
\end{center}

Therefore, if we take a non endpoint $z_i$ in the path $[a^{e_{i_{k(i)}}}_{i_{k(i)}},a^{e_{i_{1}}}_{i_1}](\Sigma)$, then when we travel $D_{w_i}$ from $z_i$ we obtain the signed word $w_i$ and so $(\Sigma,D_w)$ is a realization of $w$.   

The proof of the following proposition is not the main goal of this paper so we will omit it.

\begin{proposition}
\label{propo} For every realization $(\Sigma,\gamma_w)$ of a signed Gauss paragraph $w$ there exist a virtual realization $K_w$ of $w$ such that $(\Sigma,\gamma_w)$ is a surface realization of $K_w$. Moreover, $(\Sigma,\gamma_w)$ is stably geotopic to $(\widetilde{\Sigma},\gamma_{\widetilde{w}})$ if and only if $K_w$ is isomorphic to $K_{\widetilde{w}}$ if and only if $w$ is isomorphic to $\widetilde{w}$.
\end{proposition}

Let $w$ be a signed Gauss paragraph and let $(\Sigma,\gamma_w)\in \mathcal{D}$ be a realization of $w$. Consider the set $\mathcal{F}^{(n)}_{\Sigma}$ of all triangles of the $n^{th}$-barycentric subdivision of $\Sigma$ which intersect the curve $\gamma^{(n)}_w$. Then $\Omega^{(n)}_{\Sigma} = \bigcup_{\sigma \in \mathcal{F}^{(n)}_{\Sigma}} \sigma$ is an orientable and connected surface. The Figure \ref{reali1} describes the construction of $\Omega^{(2)}_{\Sigma}$ at a neighborhood of the crossing point $a_t$. 

\begin{theorem}
\label{min} With the notation above. For every $n\geq 2$, $(\Omega^{(n)}_{\Sigma},\gamma^{(n)}_w)\in \mathcal{D}$ is a realization of $w$ stably geotopic to $(\Sigma,\gamma_w)$. Moreover, 
\begin{enumerate}
\item $lim_{k\rightarrow \infty} \Omega^{(k)}_{\Sigma}= \gamma_w$, and
\item if $\delta^{(k)}_{\Sigma}$ is a component of the boundary, $\partial_{\Omega^{(k)}_{\Sigma}}$, of $\Omega^{(k)}_{\Sigma}$, then $lim_{k\rightarrow \infty} \delta^{(k)}_{\Sigma}$ is a closed walk in $\left|\gamma_w\right|$, where $\left|\gamma_w\right|$ is the graph $\gamma_w$ without orientation, and this corresponds to an only Carter's circle. Therefore, the number of components of the boundary of any $\Omega^{(k)}_{\Sigma}$ is equal to $b(w)$.
\end{enumerate}
\end{theorem}
 \begin{proof}
We will prove $(2)$. From (1) and the fact that for every $k$ $\delta^{(k)}_{\Sigma}$ is a closed $PL$-circle, then $\delta_{\Sigma}=lim_{k\rightarrow \infty} \delta^{(k)}_{\Sigma} $ is a closed walk in $\left|\gamma_w\right|$. 

We orient each component of the boundary of $\Omega^{(k)}_{\Sigma}$ in such way that $lk(a_{t},\Sigma^{(k)})$ has the clockwise orientation. Figure \ref{reali1} also describes the orientation of $lk(a_t,\Sigma^{(2)})$.
      
\begin{figure}[ht]
\begin{center}
\includegraphics[scale=0.12]{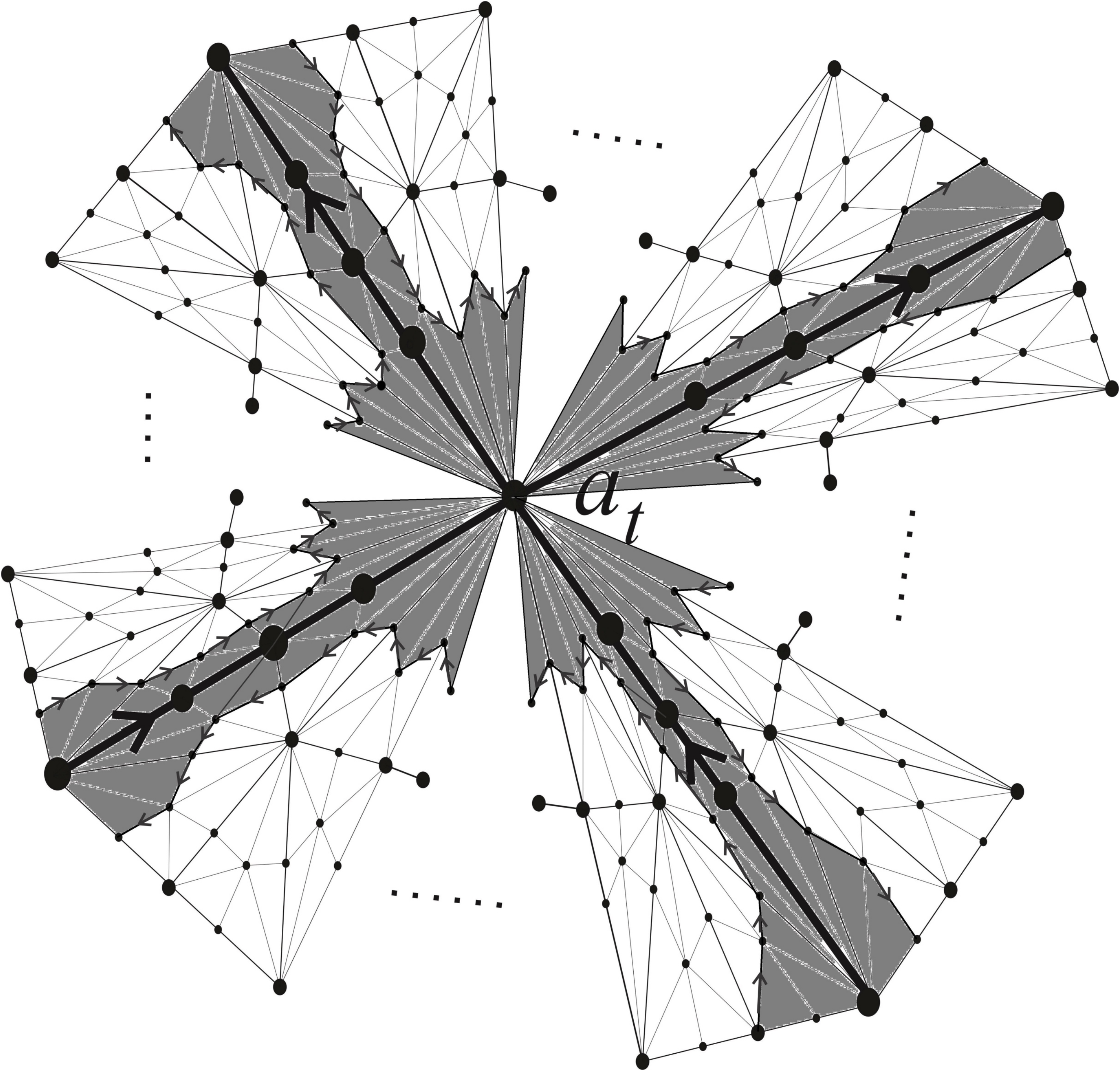}
\caption{Orientation of the $lk(a_i,\Sigma^{(2)})$}
\label{reali1}
\end{center}
\end{figure} 

Let $+\delta_{\Sigma}$ be the closed walk $\delta_{\Sigma}$ with the orientation inherit from the set $\{\delta^{(k)}_{\Sigma}\}_{k=2,3,...,\infty}$.  

We can infer, from Figure \ref{reali1}, that any component $\delta^{(k)}_{\Sigma}$ of $\partial_{\Omega^{(k)}_{\Sigma}}$ always turns ``left'' at the closed star, $st(a_i,\Sigma^{(k)})$, that it meets on the way. On the other hand, we know that each component, $\gamma_{w_i}$, of $\gamma_w$ has the form: 
\begin{center}
$\gamma_{w_{i}}=[a^{e_{i_{1}}}_{i_1},a^{e_{i_{2}}}_{i_2}](\Sigma)+[a^{e_{i_{2}}}_{i_2},a^{e_{i_{3}}}_{i_3}](\Sigma)+\cdots +[a^{e_{i_{k(i)}}}_{i_{k(i)}},a^{e_{i_{1}}}_{i_1}](\Sigma)$,
\end{center}
$i=1,2,...,n$, so $+ \delta_{\Sigma}$ can be described as a word, $E(+\delta_{\Sigma})$, in the alphabet $E(w)$ with the label of the paths of $\gamma_w$ that it meets on the way following the rule that in each crossing point $+ \delta_{\Sigma}$ always turns to ``left''. For example, if $-[a^{-1}_t,a^{e_{i_{j}}}_{i_{j}}](\Sigma)$ occurs in $E(+\delta_{\Sigma})$, then when $\delta_{\Sigma}$ turns to ``left'' at the crossing point $a_t$, it travels the path $[a^{e_{i_{k}}}_{i_{k}},a_t](\Sigma)$ in the contrary orientation of $\gamma_w$, so $-[a^{e_{i_{k}}}_{i_{k}},a_t](\Sigma)$ has to occur in $E(+\delta_{\Sigma})$, if we continue this argument, and considering the remaining three cases, we prove that $E(+\delta_{\Sigma})$ correspond to a Carter's circle. 

For every $k\neq l$, $k,l \geq 2$, $\partial _{\Omega^{(k)}_{\Sigma}}$ and $\partial _{\Omega^{(l)}_{\Sigma}}$ have the same cardinal, so let $\delta^{(k,1)}_{\Sigma}$,...,$\delta^{(k,m)}_{\Sigma}$ be the components of the boundary of $\Omega^{(k)}_{\Sigma}$ and let $\delta^{i}_{\Sigma}=lim_{k\rightarrow \infty }\delta^{(k,i)}_{\Sigma}$. Since, every edge of $\gamma_w$ is traveled in both directions by one or at most two different closed walks of $\{+\delta^{1}_{\Sigma},...,+\delta^{m}_{\Sigma}\}$, then $\bigcup^{m}_{i=1}E(+\delta^{i}_{\Sigma})=E(w)$, hence $\{E(+\delta^{1}_{\Sigma}),...,E(+\delta^{m}_{\Sigma})\}$ is a maximal set of Carter's circles. 
\end{proof}

For every $k\neq l$, $k,l \geq 2$, $\partial _{\Omega^{(k)}_{\Sigma}}$ and $\partial _{\Omega^{(l)}_{\Sigma}}$ have the same cardinality, then we can choose any $k\geq 2$ and define $S_{\Omega^{(n)}_{\Sigma}}$ as the closed surface obtained by gluing disc to the boundary of any one of the surfaces $\Omega^{(n)}_{\Sigma}$, then $(S_{\Omega^{(n)}_{\Sigma}},\gamma_w)\in \mathcal{D}$ is a realization of $w$ isomorphic to $(\Sigma,\gamma_w)$ and the genus of $S_{\Omega^{(n)}_{\Sigma}}$ is at most the genus of $S_{\Sigma}$. 

\begin{corollary}
The genus of $S_{\Omega^{(n)}_{\Sigma}}$ is $\frac{n+2-b(w)}{2}$. Therefore, $w$ is geometric if and only if $b(w)=n+2$.  
\end{corollary}
\begin{proof}
It is well known that, $2-2g(S_{\Omega^{(n)}_{\Sigma}})=\chi(\Omega^{(n)}_{\Sigma})+\left|\partial_{\Sigma}\right|$, but $\left|\partial_{\Sigma}\right|=b(w)$, and so $2-2g(S_{\Omega^{(n)}_{\Sigma}})=\chi(\Omega^{(n)}_{\Sigma})+b(w)$. Besides, $\gamma_w$ is a retract of deformation of $\Omega^{(n)}_{\Sigma}$, therefore $\chi(\Omega^{(n)}_{\Sigma})=\chi(\gamma_w)=n-2n=-n$, hence $2-2g(S_{\Omega^{(n)}_{\Sigma}})=-n+b(w)$.
\end{proof}

A \textit{minimal realization} of a signed Gauss paragraph $w$ is an element $(\Sigma,\gamma)$ of $\mathcal{D}$, where $\Sigma$ has the minimum genus among all the realizations of $w$. 

A direct consequence of the previous corollary is that for every $(\Sigma,\gamma_w)$ realizations of $w$, $(S_{\Omega^{(k)}_{\Sigma}},\gamma_w)$ is the minimal realization of $w$ up to homomorphisms.

\subsection{Primitive curves}

Let $w$ be a signed Gauss word, $(\Sigma,\gamma_w)$ a realization of $w$ and $a_t$ a crossing point of $\gamma_w$. For every $k\geq 2$, we define the oriented $PL$-normal curve of two components $(\Sigma^{(k)},\gamma^{(k)}_{w^{a_t}})$ by the process shown in the Figure \ref{reali2}.
\begin{figure}[ht]
\begin{center}
\includegraphics[scale=0.23]{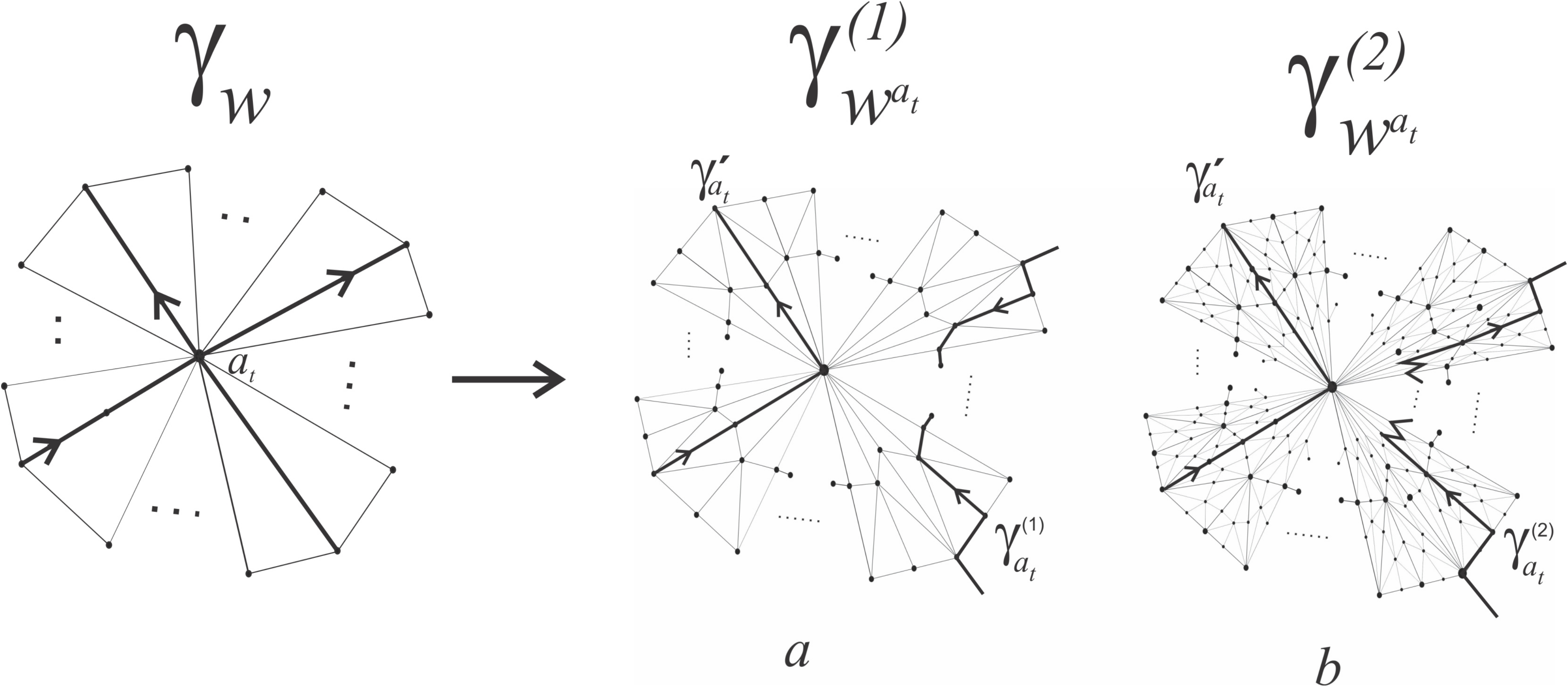}
\caption{Construction of $\gamma^{(1)}_{a_t}$ and $\gamma^{(2)}_{a_t}$}
\label{reali2}
\end{center}
\end{figure} 
 
The construction described in Figure \ref{reali2} is called \textit{splitting the crossing $a_t$}. We use $\gamma^{(k)}_{w^{a_t}}=\gamma'_{a_t} \wedge \gamma^{(k)}_{a_t}$

\begin{lemma}
\label{homolo1} For every $k,l$ the curves $\gamma^{(k)}_{a_t}$ and $\gamma^{(l)}_{a_t}$ are homologus and $lim_{k\rightarrow \infty}\gamma^{(k)}_{w^{a_t}}$ is homologue to $\gamma_w$.
\end{lemma}

The proof of the following theorem can be found in \cite{CaEl}.

\begin{theorem}
Let $w$ be a signed Gauss word and let $(S_w,\gamma_w)$ be a minimal realization of $w$. Let $\varphi_w$, $\varphi_{a_1}$,...,$\varphi_{a_n}$ be the homology classes represented by $\gamma_w$, $\gamma^{(k)}_{a_1}$,...,$\gamma^{(k)}_{a_n}$, respectively. Then $H_1(S_w,\mathbb{Z})$ is spanned by $\{\varphi_w,\varphi_{a_1},...,\varphi_{a_n}\}$.
\end{theorem} 

Let $w$ be a signed Gauss word with realization given by $(\Sigma,\gamma_w)$. Without loss of generality we may assume that $w=a_{t}\widetilde{w}_{t}a^{-1}_{t}w_{t}$, where $\widetilde{w}_{t}$ and $w_{t}$ are sub-sequences of $w$. Then $\{\widetilde{w}_{t},w_{t}\}$ is a signed Gauss paragraph of any $\gamma^{(k)}_{w^{a_t}}$, $k\geq 2$.

Now, let $w=\{w_{1},w_{2}\}$ be a signed Gauss paragraph and let $(S_{w},\gamma_{w}=\lambda \cup \sigma)$ its minimal realization. Take a common crossing $a_{1}$, of $\lambda$ and $\sigma$ and consider ordinary vertices $a'_{n+1}\in \lambda$, $a''_{n+1}\in \sigma$ in $st(a_1,\Sigma^{(2)})$. We  can write $w$ as $\{a_{1}\widetilde{w}_{1},a^{-1}_{1}\widetilde{w}_{2}\}$ and so, we define  the  \textit{join} of $\lambda$ and $\sigma$, with respect to $D_{a_1}=st(a_1,\Sigma^{(2)})$,  as the oriented $PL$-normal curve $(S_{w},\widehat{\gamma}_{w}=\lambda \vee \sigma)$, which is a realization of $\widehat{w}=a_{1}\widetilde{w}_{1}a_{n+1}a^{-1}_{1}\widetilde{w}_{2}a^{-1}_{n+1}$, obtained by the identification, $a'_{n+1}=a''_{n+1}$, of $\lambda$ and $\sigma$ described in the Figure \ref{suma}. 
\begin{figure}[ht]
\begin{center}
\includegraphics[scale=0.45]{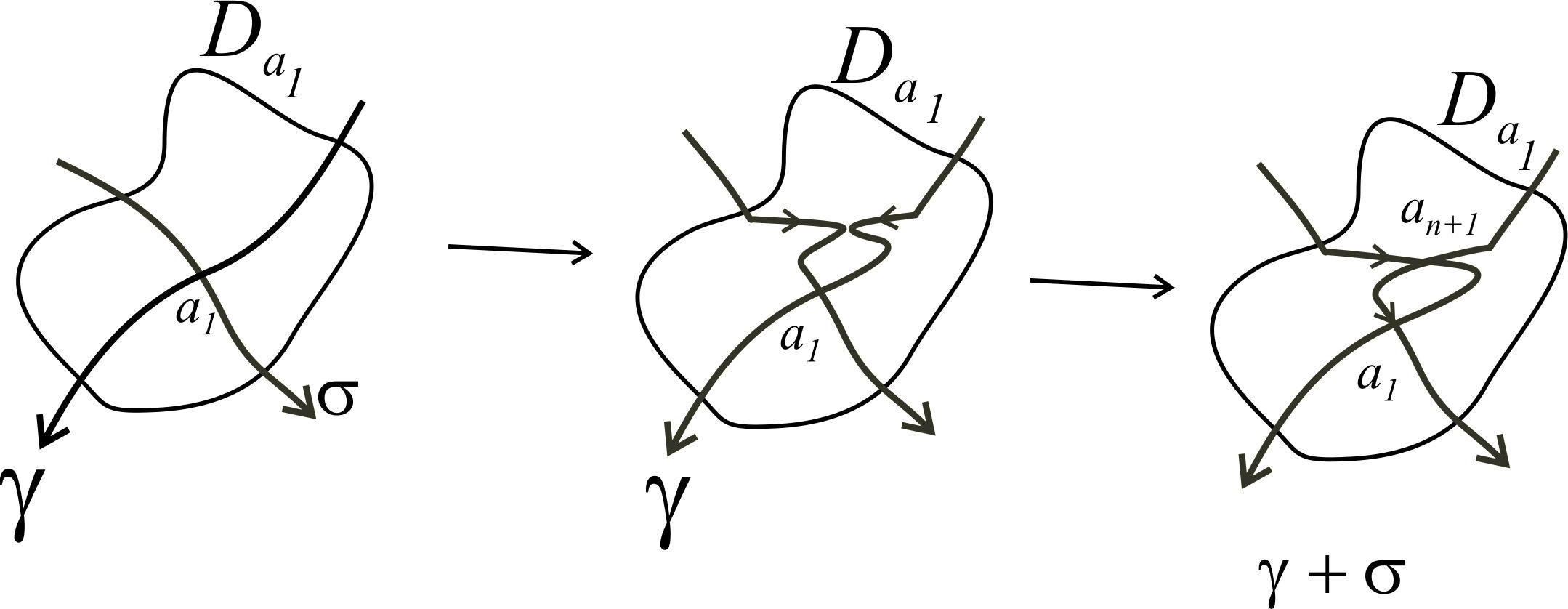}
\caption{A Join of a oriented normal curve of two components}
\label{suma}
\end{center}
\end{figure}

\begin{theorem}
With the notation above. If $w$ is a signed Gauss paragraph, then $(S_{w},\gamma_{w})$ is a minimal realization of $w$ if and only if so is $(S_{w},\widehat{\gamma}_{w})$ for $\widehat{w}$. 
\end{theorem}

\begin{proof}
Let's consider minimal realizations $(S_{w},\gamma_{w}=\lambda \cup \sigma)$ and $(S_{\widehat{w}}, \gamma_{\widehat{w}})$ of $w$ and $\widehat{w}$, respectively. The behavior of the boundary of $\Omega^{(2)}_{S_w}$ and $\Omega^{(2)}_{S_{\widehat{w}}}$ at a neighborhood that only contains the vertices $a_1$ and $a_{n+1}$ in the respectively surfaces $S_{w}$ and $S_{\widehat{w}}$ is describing in the Figure \ref{grasu}. Then $b(\widehat{w})=b(w)-1$. Therefore 
\begin{center}
$g(S_{\widehat{w}})=\frac{(n+1)+2+b(\widehat{w})}{2}=\frac{(n+1)+2+b(w)-1}{2} =g(S_{w})$.
\end{center}

\begin{figure}[ht]
\begin{center}
\includegraphics[scale=0.45]{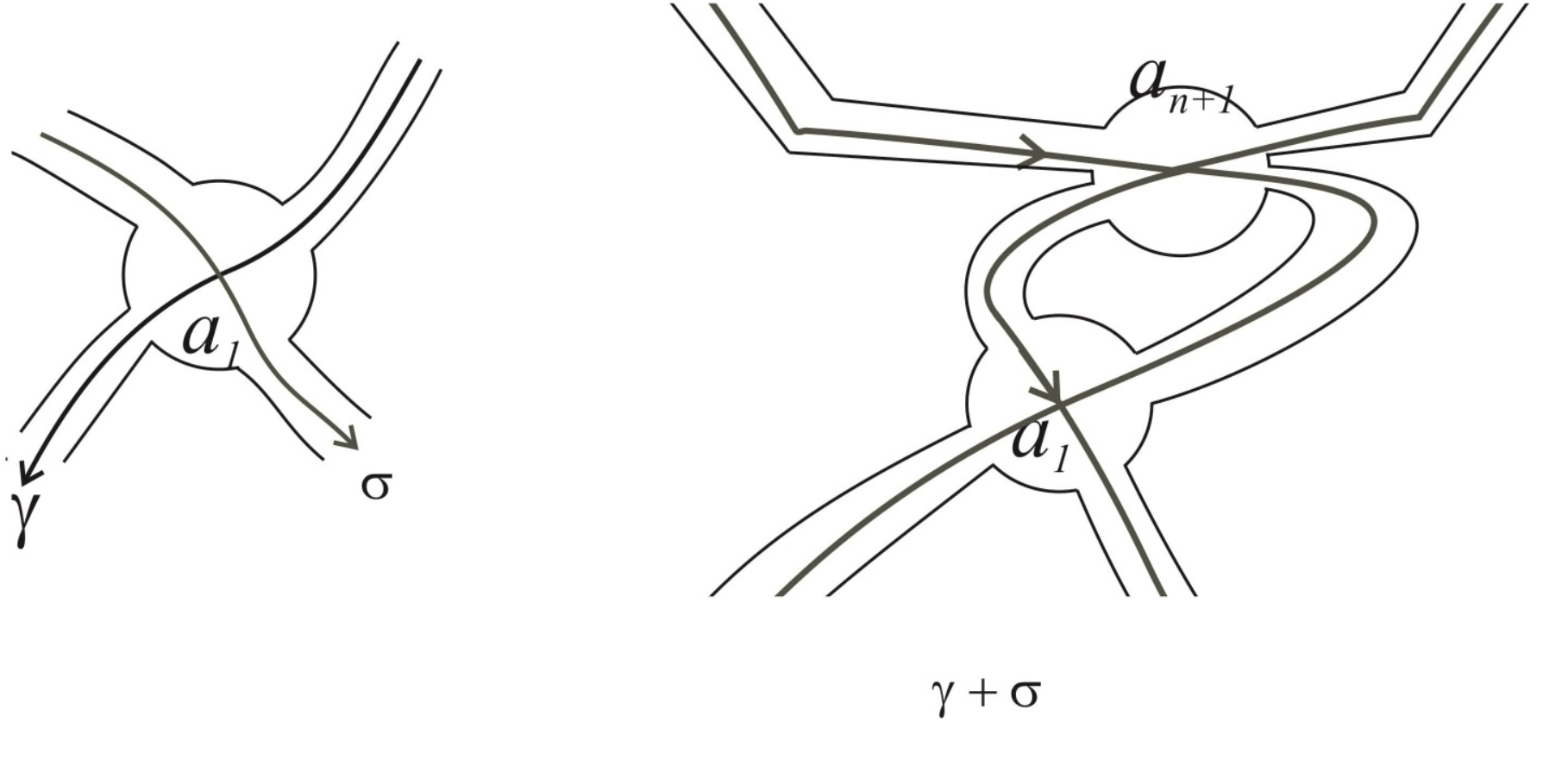}
\caption{A local representations of the minimal realization of a join of a oriented normal curve of two components}
\label{grasu}
\end{center}
\end{figure}
\end{proof}

\section{Intersection number of $PL$-oriented normal curves}
A very useful characteristic of the relative position of two curves on an orientable surface is their intersection number. This number has been defined in \cite{FoMa} and \cite{Fu}. In this section we recall such definition and we prove some of its most relevant properties.
 
Let  $\gamma_1$ and $\gamma_2$ be two $PL$-normal curves on a surface $\Sigma$, we say that they are \textit{transverse} denoted $\gamma_1 \bot \gamma_2$, if $(\Sigma,\gamma_1 \cup \gamma_2)$ is in $\mathcal{D}$, in other words, if $\gamma_1 \cup \gamma_2$ is an oriented normal curve of two components of $\Sigma$.
  
\begin{definition}
Let $\gamma_1$ and $\gamma_2$ be two oriented $PL$-normal curves on a surface $\Sigma$, with $\gamma_1 \bot \gamma_2$. For every $c \in \gamma_1 \cap \gamma_2$, let $\epsilon^{\gamma_1 \gamma_2}_{c} \in \{1,-1\}$ be defined as follows: $\epsilon^{\gamma_1 \gamma_2}_{c}=1$ (or $\epsilon^{\gamma_1 \gamma_2}_{c}=-1$) if when $\gamma_1$ travels the crossing point $c$ positively (negatively).
\end{definition}

From the previous definition, $\epsilon^{\gamma_1 \gamma_2}_{c}=-\epsilon^{\gamma_2 \gamma_1}_{c}$, for every $c\in \gamma_1 \cap \gamma_2$.

\begin{definition}[Intersection pairing]
Let $\gamma_1$ and $\gamma_2$ to be two normal curves on a surface $\Sigma$ that intersect transversally with $\gamma_1 \cap \gamma_2 = \{c_1,...,c_t\}$. 

We define the intersection pairing from $\gamma_1$ to $\gamma_2$, denoted by $\left\langle \gamma_1,\gamma_2\right\rangle$, as the sum  $\epsilon^{\gamma_1 \gamma_2}_{c_1}+ \cdots +\epsilon^{\gamma_1 \gamma_2}_{c_k}$.
\end{definition}

The following theorem summarizes some relevant properties of the intersection pairing. Its proof is a direct consequence of what we did before, so we will omit it. 

\begin{theorem}
\label{Prop} Let $\varphi_1$, $\varphi_2$ and $\varphi_3$ be oriented normal curves on a surface $\Sigma$. Suppose that they are transverses, then  
\begin{enumerate}
\item $\left\langle \varphi_1,\varphi_2 \right\rangle=-\left\langle \varphi_2,\varphi_1 \right\rangle$,
\item $\left\langle \varphi_1,\varphi_2 +\varphi_3\right\rangle=\left\langle \varphi_2,\varphi_1 \right\rangle+\left\langle \varphi_1,\varphi_3 \right\rangle$ and
\item for every $n \in \mathbb{Z}$, $\left\langle n \varphi_1,\varphi_2 \right\rangle=n \left\langle \varphi_1,\varphi_2 \right\rangle$
\end{enumerate}
\end{theorem} 

A proof from the point of view of the differential topology of the next theorem can be found in \cite{GiPo}, but here we present a combinatorial one. 

\begin{theorem}
Let $\gamma_1$ and $\gamma_2$ to be two oriented normal curves on a surface $\Sigma$. Then, there exist an oriented normal curve $\gamma^{p}_1$, homologous to $\gamma_1$, such that $\gamma^{p}_1 \bot \gamma_2$. Moreover, $\gamma^{p}_1 \bot \gamma_1$ and $\left\langle \gamma_1,\gamma^{p}_1\right\rangle=0$.
\end{theorem}
\begin{proof}
Choose a crossing point $a_i$ of $\gamma_1$ and let's consider  $st(a_i,\Sigma)$, $st(a_i,\Sigma{{(1)}})$ and $st(a_i,\Sigma^{(2)})$ be the closed stars of the vertex $a_i$ in $\Sigma$, $\Sigma^{(1)}$ and $\Sigma^{(2)}$, respectively, see Figure \ref{st1}-$a$. Now, apply the modification, on $\gamma_1$, showed in the Figure \ref{st1}-$b$. 

\begin{figure}[ht]
\begin{center}
\includegraphics[scale=0.2]{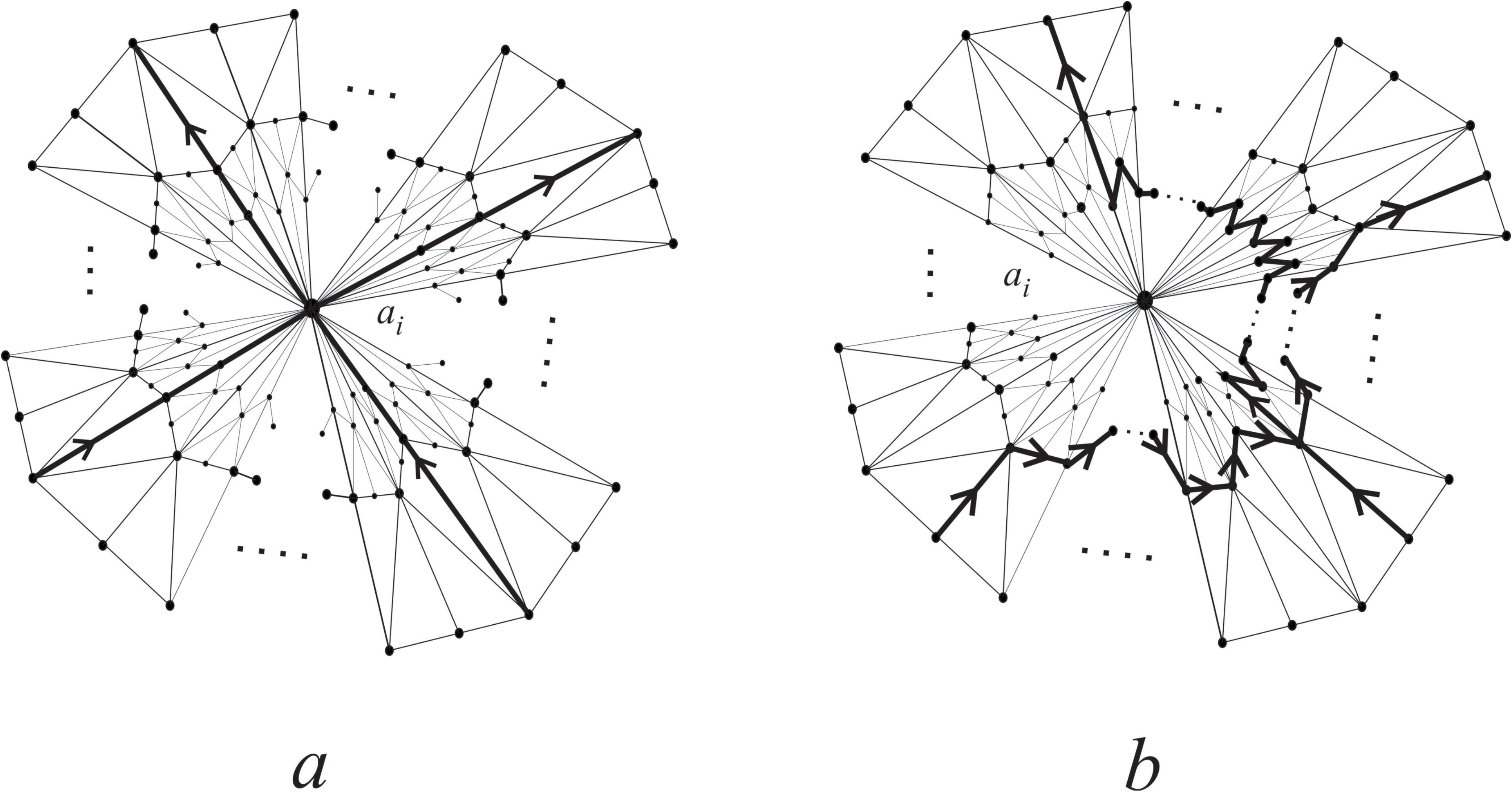}
\caption{Changing crossing point}
\label{st1}
\end{center}
\end{figure}

Now, if $v_j$ is an ordinary vertex of $\gamma_1$, that is not a crossing, then we proceed as above, but we only modify the part of $\gamma_1$ inside of closed star $\overline{st}(v_j)$ of $\Sigma$ as Figure \ref{st2} shows.

\begin{figure}[ht]
\begin{center}
\includegraphics[scale=0.2]{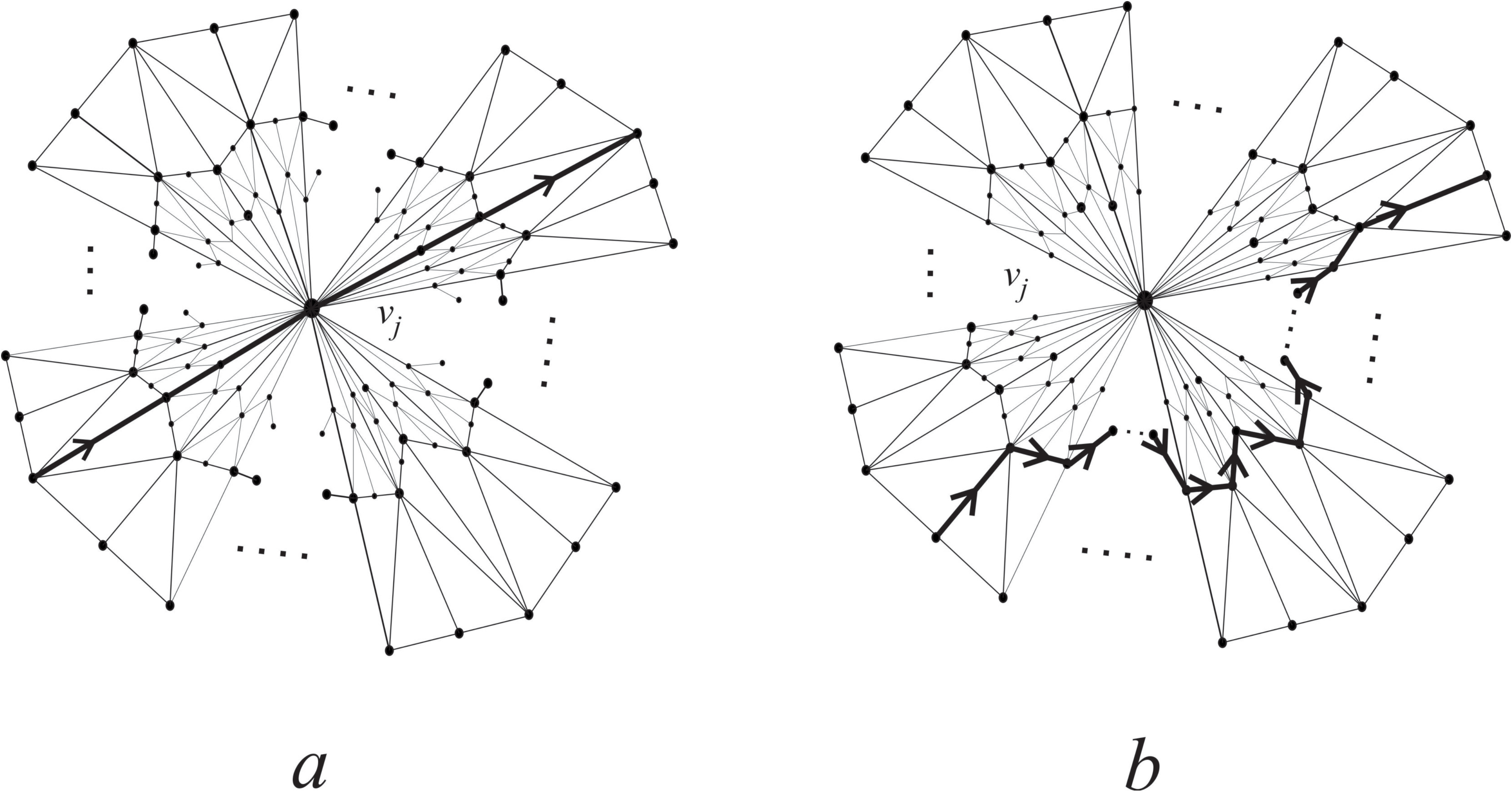}
\caption{Changing vertex point}
\label{st2}
\end{center}
\end{figure}

The resultant oriented $PL$-normal curve obtained by applying the above modifications on all the vertices of $\gamma_1$ is denoted by $\gamma^{p}_1$. As we see, this curve is homologous to $\gamma_1$ and $\gamma^{p}_{1} \bot \gamma_i$, $i=1,2$.  

Suppose that $a_1$,...,$a_{n}$ are the crossings points of $\gamma_1$. If $a'_i$ and $a''_i$ denote the two intersections of $\gamma_1$ and $\gamma^{p}_1$ at $st(a_i,\Sigma)$, then $\epsilon^{\gamma_1 \gamma^{p}_1}_{a'_i}+\epsilon^{\gamma_1 \gamma^{p}_1}_{a''_i}=0$. As a consequence $\left\langle \gamma_1,\gamma^{p}_1\right\rangle=0$.

\end{proof}

The curve $\gamma^{p}_1$ given in above theorem is called the \textit{parallel normal curve} associated to $\gamma_1$.

\begin{definition}[Intersection number]
Let $\varphi_1$ and $\varphi_2$ two homology classes in $H_1(\Sigma,\mathbb{Z})$ with representative transverses normal curves $\gamma_1$ and $\gamma_2$, respectively. We define the intersection number between $\varphi_1$ and $\varphi_2$, denoted $\left\langle \varphi_1,\varphi_2 \right\rangle$, as $\left\langle \gamma_1, \gamma_2 \right\rangle$.
\end{definition}

The following theorem proves the well definition of the homology intersection.

\begin{theorem}
If $\gamma_i$ and  $\gamma'_i$, $i=1,2$, are curves on a surface $\Sigma$, such that $\gamma_i \bot \gamma'_i$, $i=1,2$, $\gamma_1$ is homologous to $\gamma_2$ and $\gamma'_1$ is homologous to $\gamma'_2$. Then, $\left\langle \gamma_1,\gamma'_1\right\rangle=\left\langle \gamma_2,\gamma'_2\right\rangle$.
\end{theorem}
\begin{proof}
Let's take curves $\gamma \bot \gamma_i$, $i=1,2$, on a surface $\Sigma$, with $\gamma_1$ homologous to $\gamma_2$, and we will prove that $\left\langle \gamma,\gamma_1 \right\rangle=\left\langle \gamma,\gamma_2 \right\rangle$.

Since $\gamma_1$ and $\gamma_2$ are homologous curves, then there exists a $2$-chain $\mu$ in $\Sigma$ such that $\gamma_1=\gamma_2 + \partial(\mu)$. We write $\mu= \sum^{k}_{i=1} \delta_i $, where $\delta_i$ is a triangle, $i=1,2,...,k$. Since, $\partial(\delta_i)$ is the border of a disc in $\Sigma$, then from the closed Jordan curve theorem, for every oriented normal curve $\lambda$ transverses to $\partial(\delta_i)$ we have that $\left\langle \lambda, \partial(\delta_i)\right\rangle=0$, so  $ \left\langle \gamma_1, \partial(\Sigma_i)\right\rangle=0$ for every $i=1,2,...,k$, therefore $\left\langle \gamma, \gamma_1 \right\rangle=\left\langle \gamma,\gamma_2 \right\rangle$.
\end{proof}

A direct consequence of the previous theorem is the following corollary.
\begin{corollary}
Suppose that  $\gamma_1$ and $\gamma_2$ are two oriented normal curves on a surface $\Sigma$ that intersect transversely. If $\gamma_2$ is null-homologous, then $\left\langle \gamma_1,\gamma_2\right\rangle=0$.  
\end{corollary} 

 \begin{theorem}
Let $\varphi \in H_1(\Sigma,\mathbb{Z})$. If for every $\lambda \in H_1(\Sigma,\mathbb{Z})$, $\left\langle \varphi, \lambda \right\rangle=0$, then $\varphi$ is null-homologue.
\end{theorem}
\begin{proof}
Every oriented, connected and closed surface $\Sigma$ of genus $g$ can be represented as in Figure \ref{genecur}, where the curves $a_i$, $b_i$, $i=1,2,...,g$, generate the first homology group $H_1(\Sigma,\mathbb{Z})$. Moreover, they satisfy: $\left\langle a_i,a_j\right\rangle=\left\langle b_i,b_j\right\rangle=0$ and $\left\langle a_i,b_i\right\rangle=1$, for every $i,j=1,2,...,g$.

\begin{figure}[ht]
\begin{center}
\includegraphics[scale=0.3]{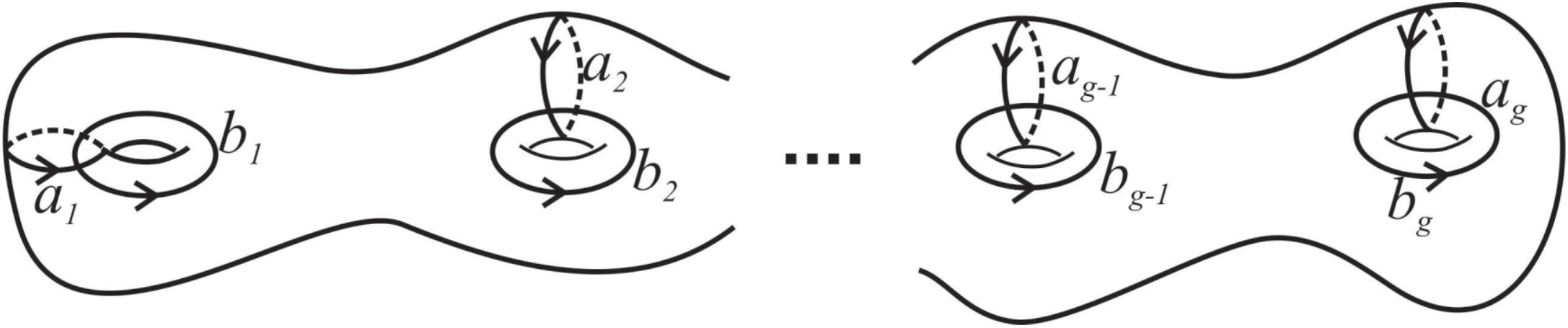}
\caption{Generators of the first homology groups of $\Sigma$}
\label{genecur}
\end{center}
\end{figure}

Let $\varphi \in H_1(\Sigma,\mathbb{Z})$, then $\varphi = \sum g_i a_i +\sum h_j b_j$, with $\left\langle \varphi , a_j\right\rangle=h_j$ and $\left\langle \varphi , b_i\right\rangle=g_i$. Therefore, $\varphi$ is null-homologous if and only if   $\left\langle \varphi , a_j\right\rangle=0$ and $\left\langle \varphi , b_i\right\rangle=0$.

\end{proof}

\section{A short solution of the signed Gauss word}

Let $w$ be a signed Gauss word and let $(S_w,\gamma_w)$ be a minimal realization of $w$. Chose a $k\geq 2$ and denote by $\varphi_w$, $\varphi_{a_1}$,...,$\varphi_{a_n}$ the homology classes represented by the oriented $PL$-normal curves $\gamma_w$, $\gamma^{(k)}_{a_1}$,...,$\gamma^{(k)}_{a_n}$ in $H_1(S_w,\mathbb{Z})$, respectively.

\begin{theorem}
With the notation above. $H_1(S_w,\mathbb{Z})=(0)$ if and only if, for every $i,j=1,2,...,n$, $\alpha_{a_j}(w)=\left\langle \varphi_w, \varphi_{a_i}\right\rangle=0$ and $\beta_{a_i a_j}(w)=\left\langle \varphi_{a_j}, \varphi_{a_i}\right\rangle=0$.
\end{theorem}  
\begin{proof}
Let $\gamma , \lambda \in H_{1}(\Sigma_{w},\mathbb{Z})$, since $\{\varphi_{w},\varphi_{a_1},...,\varphi_{a_n}\}$ spans $H_{1}(\Sigma_{w},\mathbb{Z})$, then there exist scalars $u,u_{1},...,u_{n}, v,v_{1},...,v_{n}$ such that $\lambda = u \varphi_{w}+u_{1}\varphi_{a_1}+\cdots + u_{n}\varphi_{a_n}$, and $\varphi = v \varphi_{w}+v_{1}\varphi_{a_1}+\cdots + u_{n}\varphi_{a-n}$. So,
\begin{center}
$\left\langle \gamma,\lambda\right\rangle=v\sum u_{i}\alpha_{a_i}(w)-u\sum v_{i}\alpha_{a_i}(w)+\sum \sum v_{i}u_{j} \beta_{a_i a_j}(w)$.
\end{center}
If we suppose that $\alpha_{a_j}(w)=0$ and $\beta_{a_i a_j}(w)=0$, for every $i,j=1,2,...,n$, then $\left\langle \gamma,\lambda\right\rangle=0$, hence $\gamma$ es null-homologous and so $H_{1}(\Sigma_{w},\mathbb{Z})=(0)$.
Reciprocally, if $H_{1}(\Sigma_{w},\mathbb{Z})=(0)$, then $\alpha_{a_j}(w)=\left\langle \varphi_{w} , \varphi_{a_j}\right\rangle=0$ and $\beta_{a_i a_j}(w)=\left\langle \varphi_{a_i}, \varphi_{a_j}\right\rangle=0$, for every $i,j=1,2,...,n$.
\end{proof}

The proof of the following theorem is straightforward from the definitions of the signed Gauss paragraphs and intersection pairing, so we will omit it. 

\begin{theorem}
Let $(\Sigma,\gamma=\gamma_1 \cup \gamma_2)$ be an oriented $PL$-normal curve with signed Gauss paragraph $w=\{w_1,w_2\}$. Denote by $S(w)$ the set of all $a^{p}$ such that $a^{p}$ occurs in $w_1$ and $a^{-p}$ occurs in $w_2$, then

\begin{center}
$\left\langle \gamma_1,\gamma_2\right\rangle=\sum_{a^{p}\in S(w)} p$.
\end{center}
\end{theorem}

A direct consequence of the previous theorem is the next corollary.

\begin{corollary}
The homology intersection is an invariant under homeomorphisms.
\end{corollary}

Let $w=a_iw_{a_i}a^{-1}_i \widetilde{w}_{a_i}$ be a signed Gauss word, and let $S_{a_i}(w)$ be the set of all $a^{p}$ that occurs in $w_{a_i}$.

\begin{corollary}
For every $i=1,2,...,n$, $\alpha_{a_i}(w)=\sum_{a^{p}\in S_{a_i}(w)} p$.
\end{corollary}
\begin{proof}
We know that $\check{w}=\{w_{a_i},\widetilde{w}_{a_i}\}$ is a signed Gauss paragraph of $(S_w,\gamma^{(k)}_{w^{a_i}}=\gamma'_{a_i} \cup \gamma^{(k)}_{a_{i}})$. Besides, from Lemma \ref{homolo1}, $\left\langle \gamma^{(k)}_w,\gamma^{(k)}_{a_i}\right\rangle=\left\langle \gamma'_{a_i},\gamma^{(k)}_{a_i}\right\rangle$, then 
\begin{center}
$\alpha_{a_i}(w)=\sum_{a^{p}\in S(\check{w})} p$.
\end{center}
But $S(\check{w})=S_{a_i}(w)$.
\end{proof}

From Lemma \ref{homolo1}, we do not lost generality if for some $k\geq 2$, $\gamma_j = \gamma^{(k)}_{a_j}$, $j=1,2,...,n$. 

\begin{proposition}

Let $\overline{S}_{a_i}(w)=S_{a_i}(w)\cup \{a_i,a^{-1}_i\}$, and $S^{-1}_{a_j}(w)=\{a^{-p} \mid a^{p}\in S_{a_j}(w)\}$. Then
\begin{center}
$\beta_{a_{i}a_{j}}(w)=\sum\limits_{a^{e}\in \overline{S}_{a_{i}}(w) \cap S^{-1}_{a_{j}}(w)} e$, 
\end{center}
for every $i,j=1,2,...,n$.

\end{proposition} 
\begin{proof}

Let $a^{e_{r}}_{r} \in \overline{S}_{a_i}(w)\cap S^{-1}_{a_j}(w)$. Consider the following cases:

\textbf{Case 1:} Suppose that $a^{e_{r}}_{r}\neq a^{\pm 1}_{i}$. Then $a_{r}$ represents a common crossing point of $\gamma_{i}$ and $\gamma^{p}_{j}$, and moreover $\epsilon^{\gamma_i \gamma^{p}_j}_{a_r}$

\textbf{Case 2:} Consider $a^{e_r}_r=a_i$, then $a^{-1}_i\in S_{a_j}(w)$ and so, $\gamma^{(p)}_j$ travels the crossing point $a_i$ negatively, see Figure \ref{cruinf2}. Figure \ref{cruinf2} describes the constructions of $\gamma_i$. from that, we conclude that $\epsilon^{\gamma_i \gamma^{p}_j}_{a_i}=-1$.
\begin{figure}[ht]
\begin{center}
\includegraphics[scale=0.6]{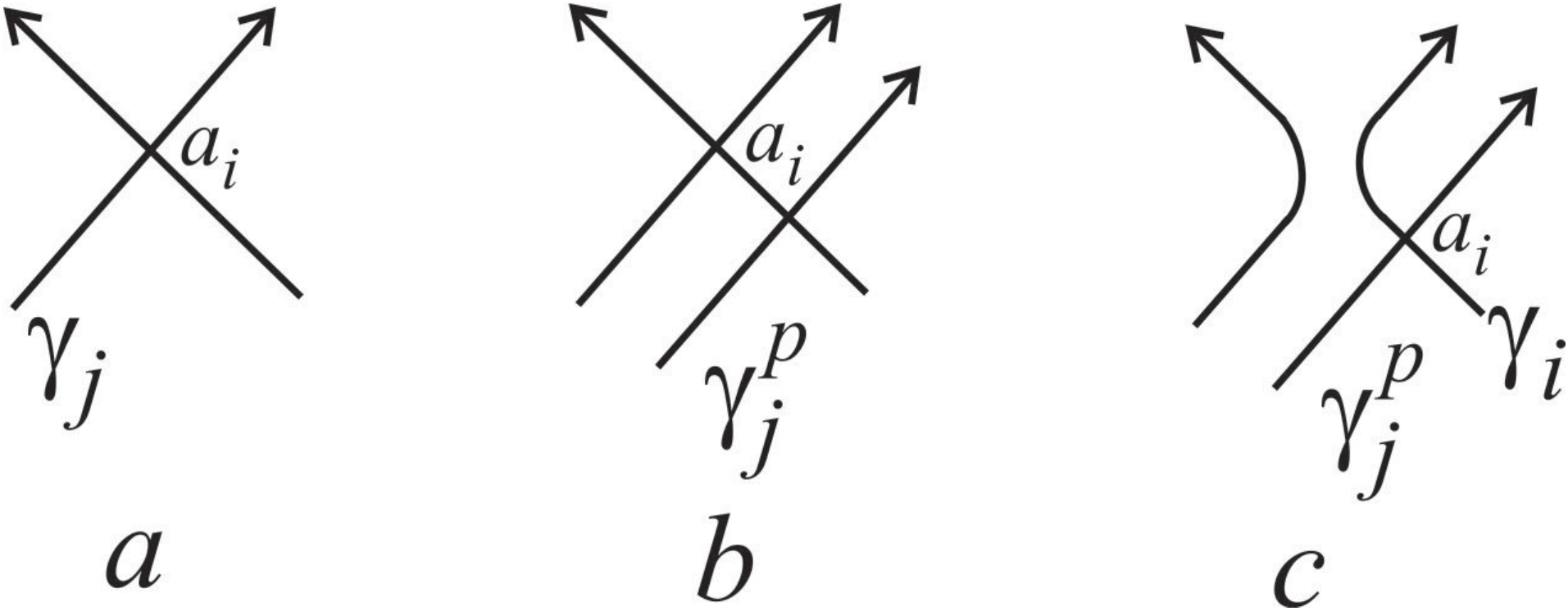}
\caption{Homology product}
\label{cruinf2}
\end{center}
\end{figure}

From the Figure \ref{cruinf2}, $a_i \in S(w_{a_i a_j})$.

\textbf{Case 3:} If $a^{e_r}_r=a^{-1}_i$, we proceed as in Case $2$, and so, $\epsilon^{\gamma_i \gamma^{p}_j}_{a_i}=1$.
\end{proof}

\end{document}